\documentclass[a4paper]{article}

\usepackage[utf8]{inputenc}
\usepackage[T1]{fontenc}
\usepackage[ngerman,english]{babel}
\usepackage{lmodern}

\usepackage[pdftex]{graphicx}
\usepackage{latexsym}
\usepackage{amsmath,amssymb,amsthm}
\usepackage{hyperref}

\usepackage{enumerate,enumitem}
\usepackage{graphicx}
\usepackage{xfrac}
\usepackage{tabularx,booktabs,multirow}
\usepackage[usenames,dvipsnames]{xcolor}
\usepackage{array}
\usepackage{mathtools}
\usepackage{algorithm} 
\usepackage{algorithmic}

\usepackage{authblk}

\DeclareGraphicsExtensions{.pdf,.png,.eps}
\graphicspath{{figs/}}

\newtheorem{theorem}{Theorem}

\newtheorem{proposition}[theorem]{Proposition}

\newtheorem*{conjecture*}{Conjecture}
\newtheorem*{observation*}{Observation}

\newtheorem*{claim*}{Claim}

\numberwithin{equation}{section} 

\newcommand{\R}{\mathbb{R}} 
\newcommand{\N}{\mathbb{N}} 

\newcommand{\GG}{\ensuremath{\mathcal{G}}}
\newcommand{\HH}{\ensuremath{\mathcal{H}}}

\newcommand{\PP}{\ensuremath{\mathcal{P}}}
\newcommand{\RR}{\ensuremath{\mathcal{R}}}

\newcommand{\RRR}{\ensuremath{\hat{R}}}

\newcommand*{\medcup}{\mathbin{\scalebox{1.2}{\ensuremath{\cup}}}}

\begin{document}

\title{Lower Bound on the Size-Ramsey Number of Tight Paths}

\author{Christian Winter}
\affil{University of Hamburg, Hamburg, Germany; and Karlsruhe Institute of Technology, Karlsruhe, Germany; E-mail: \textit{christian.winter@kit.edu}}

\maketitle

\begin{abstract}
The size-Ramsey number $\RRR^{(k)}(\HH)$ of a $k$-uniform hypergraph $\HH$ is the minimum number of edges in a $k$-uniform hypergraph $\GG$ with the property that every `$2$-edge coloring' of $\GG$ contains a monochromatic copy of $\HH$.
For $k\ge2$ and $n\in\N$, a $k$-uniform tight path on $n$ vertices~$\PP^{(k)}_{n}$ is defined as a $k$-uniform hypergraph on $n$ vertices for which there is an ordering of its vertices such that the edges are all sets of $k$ consecutive vertices with respect to this order.

We prove a lower bound on the size-Ramsey number of $k$-uniform tight paths, which is, considered assymptotically in both the uniformity $k$ and the number of vertices $n$,
$\RRR^{(k)}(\PP^{(k)}_{n})= \Omega\big(\log (k)n\big)$.
\end{abstract}
\smallskip
\textbf{Keywords\ --} size-Ramsey, Ramsey theory, tight path, uniform hypergraph\\

\section{Introduction}\label{sec_intro}

For a $k$-graph $\GG=(V,E)$, i.e.\ a $k$-uniform hypergraph on a vertex set $V$ and an edge set $E\subseteq \binom{V}{k}$,
a \textit{$2$-edge coloring} of $\GG$ is a function $c\colon E(\GG)\to \{\text{red},\text{blue}\}$
that maps every edge to one of the given colors \textit{red} or \textit{blue}.
In the following we refer to such a function simply as a \textit{coloring} of $\GG$. 
We say that a $k$-graph $\GG$ has the \textit{Ramsey property} $\GG\rightarrow \HH$ for some $k$-graph $\HH$ if every coloring of $\GG$ contains a monochromatic copy of $\HH$.
The \textit{size-Ramsey number} of a $k$-graph $\HH$ is defined as 
$$\RRR^{(k)}(\HH)=\min\big\{|E(\GG)| \colon \GG\ k\text{-graph with }\GG\rightarrow \HH\big\}.$$

Size-Ramsey problems were introduced by Erd\H{o}s, Faudree, Rousseau and Schelp \cite{erdos_size} for graphs.
One of the focus points of studies on the graph case is estimating the size-Ramsey number of paths.
Beck \cite{beck83} disproved a conjecture of Erd\H{o}s \cite{erdos_pn} by showing that $\RRR^{(2)}(P_n)=O(n)$.
Since then, estimates on this number have been gradually improved, with the current best known bounds being
$\big(3.75-o(1)\big)n\le\RRR^{(2)}(P_n)\le74n$
given by Bal, DeBiasio \cite{bal_lb} and Dudek, Pralat \cite{dud_74}, respectively.
\\

Let $n,k\in\N$ with $k\ge2$.
A \textit{$k$-uniform tight path} on $n$ vertices $\PP^{(k)}_{n}$ is a $k$-graph on $n$ vertices 
for which there exists an ordering of its vertices such that every edge is a $k$-element set of consecutive vertices with respect to this order, 
two consecutive edges have precisely $k-1$ vertices in common, and there are no isolated vertices.
Equivalently, $\PP^{(k)}_{n}$ is a $k$-graph isomorphic to the hypergraph $(\{1,\dots,n\},E)$ with edge set $$E=\big\{\{i,\dots,i+k-1\}\colon i\in\{1,\dots,n-k+1\}\big\}.$$
If the uniformity is clear from the context we omit the prefix `$k$-uniform' when referring to tight paths.
\\

%
%
Research on the size-Ramsey number of hypergraphs has been substantially driven forward by Dudek, La Fleur, Mubayi and Rödl \cite{dud_general}. 
Among other results, they conjectured that the size-Ramsey number of tight paths is linear in terms of $n$. 
This conjecture was recently verified by Letzter, Pokrovskiy and Yepremyan \cite{let_ub}.

\begin{theorem}[\cite{let_ub}]\label{ub_linear}
Let $k\ge 2$ be fixed. Then
$$\RRR^{(k)}(\PP^{(k)}_{n})=O(n).$$
\end{theorem}

Regarding a lower bound on this number, the following is a simple observation.

\begin{observation*}\label{lb_trivial}
Let $n,k\in\N$, $k\ge2$. Then
$$\RRR^{(k)}(\PP^{(k)}_{n})\ge 2n-2k+1.$$
\end{observation*}
\medskip

In this paper we show an improved lower bound on the size-Ramsey number of tight~paths.

\begin{theorem}\label{lb_tight2}
Let $n\ge 7$. Then
$$\RRR^{(3)}(\PP^{(3)}_{n})\ge\tfrac{8}{3} n -\tfrac{28}{3}.$$
\end{theorem}

\begin{theorem}\label{lb_tight}
Let $k\ge4$ and $n>\frac{k^2+k-2}{2}$. Then
$$\RRR^{(k)}(\PP^{(k)}_{n})\ge \big\lceil\log_2(k+1)\big\rceil\cdot n-2k^2.$$
\end{theorem}
\medskip

Section \ref{sec_neigh} discusses some properties which are useful for the main proofs. 
In Section \ref{sec_tight} the proofs of Theorem \ref{lb_tight} and Theorem \ref{lb_tight2} are presented.
\bigskip

\section{Preliminaries}\label{sec_neigh}

Let $\GG$ be a $k$-graph and $Z\subseteq E(\GG)$ be an edge set. 
Let $\medcup Z =\{v\in e\colon e\in Z\}$ be the set of vertices that are \textit{covered} by~$Z$.
We say that the $k$-graph $(\medcup Z, Z)$ is \textit{formed} by $Z$.
Given a vertex set $W\subseteq V(\GG)$ the subhypergraph \textit{induced} by $W$ is $\GG[W]=(W,\{e\in E(\GG)\colon e\subseteq W\})$.
For $q\in\R$, $0\le q< k$, the \textit{$q$-neighborhood} of $Z$ is the edge set
$$N_{> q}(Z)=\big\{e\in E(\GG) \colon \exists e'\in Z\text{ with }|e\cap e'|> q\big\}.$$
Note that we allow $e=e'$, thus $Z\subseteq N_{>q}(Z)$ for all $0\le q< k$.
\\

For each $k$-uniform tight path $\PP$ on $n$ vertices we fix an ordering of the vertices such that each edge is a set of consecutive vertices.
We say that such an enumeration $V(\PP)=\{v_1,\dots,v_n\}$ is \textit{according to} $\PP$.
For a $k$-graph $\GG$, we define $e(\GG)=|E(\GG)|$, e.g.\ $e(\PP^{(k)}_{n})=n-k+1$.
Furthermore, let $[n]=\{1,\dots,n\}$ for $n\in\N$. For any other notation, see Diestel \cite{diestel}.


\begin{proposition}\label{prop_gen}
Let $n,k\in\N$ such that $k\ge 2$ and $n>\frac{k^2+k-2}{2}$. Let $\PP$ be a $k$-uniform tight path on $n$ vertices. 
Furthermore, let $\alpha\in\R$ such that $1\le \alpha \le k$ and $W\subseteq V(\PP)$ be a vertex set such that for every edge $e\in E(\PP)$ we have $|e\cap W|\ge \alpha$.
Then $$|W|\ge \frac{\alpha(n-k+1)}{k}.$$

\noindent In particular, if for each $e\in E(\PP)$, $|e\cap W|> \frac{k+1}{2}$, then for $n>\frac{k^2+k-2}{2}$, $$|W|> \frac{n}{2}.$$
\end{proposition}

\begin{proof}
We estimate the size of $W$ by double-counting ordered pairs $(v,e)$ consisting of a vertex $v\in W$ and an edge $e\in E(\PP)$ with $v\in e$. 
Let $\rho_{(v,e)}$ be the number of such pairs.

Considering the edges of $\PP$ it is immediate that
$$\rho_{(v,e)}\ge\alpha \cdot e(\PP)= \alpha (n-k+1).$$

Now consider the vertices in $W\subseteq V(\PP)$.
The maximum degree of the tight path $\PP$ is at most $k$, so
$$\rho_{(v,e)}\le k\cdot |W|.$$

\noindent Combining both inequations, we obtain
$$|W|\ge \frac{\alpha (n-k+1)}{k}.$$

Now consider the case that for each edge $e\in E(\PP)$ we have $|e\cap W|>\frac{k+1}{2}$, then also $|e\cap W|\ge\frac{k+2}{2}$.
Therefore we obtain for sufficiently large $n$, 
\begin{align*}
|W|\ge \frac{k+2}{2}\cdot\frac{n-k+1}{k}> \frac{n}{2}.&\qedhere
\end{align*}
\end{proof}
\medskip
	
\section{Proofs of the main results}\label{sec_tight}

\begin{proof}[Proof of Theorem \ref{lb_tight}]
Let $\GG$ be a $k$-uniform hypergraph with $\GG\rightarrow \PP^{(k)}_{n}$, i.e.\ such that every $2$-coloring contains a monochromatic $k$-uniform tight path on $n$ vertices.
We show that there are at least $\left\lceil\log_2(k+1)\right\rceil\cdot n-2k^2$ many edges in $\GG$ 
by iteratively constructing many edge-disjoint tight paths of length $n$.
Let $\lambda=\left\lceil\log_2(k+1)\right\rceil-1$, this number indicates how many iteration steps are executed.
Additionally, we define the function $q\colon\{0,\dots,\lambda\}\rightarrow \R$, 
$$q(i)=\left(1-\frac{1}{2^i}\right)(k+1),$$
which will be the parameter of the $q$-neighborhoods considered in each iteration step.
Clearly, $q$ is an increasing function and $q(i)\ge 0$ for $i\in\{0,\dots,\lambda\}$. 
For $i\le\lambda$ (or equivalently $i<\log_2(k+1)$) it can be seen that $q(i)<k$, 
which implies that the $q(i)$-neighborhood is well-defined for all $i\in\{0,\dots,\lambda\}$.
\\

As an initial step of the iteration, the Ramsey property $\GG\rightarrow \PP^{(k)}_{n}$ provides that there is some tight path on $n$ vertices in $\GG$, which we denote by $\PP_0$.
\\

\noindent From now on we proceed iteratively, so let $i= 1,\dots,\lambda$ and suppose that the iteration has been performed for all smaller values of $i$.
In each step of the iteration we construct the following:
\begin{itemize}
\item Edge sets $Z^1_i, Z^2_i\subseteq E(\PP_{i-1})$ such that $\medcup Z^1_i \cap \medcup Z^2_i=\varnothing$ and each of the sets forms a tight path in $\GG$ on precisely $\left\lfloor\frac{n}{2}\right\rfloor$ vertices.

\item A tight path $\PP_{i}$ on $n$ vertices with $E(\PP_{i})\cap N_{>q(i)}(Z^{a}_{b})=\varnothing$ for all $a\in[2], b\in[i]$.
\end{itemize}

\noindent First we construct $Z^1_i$ and $Z^2_i$ by dividing the tight path $\PP_{i-1}$ into two parts of equal length 
and considering the edge sets of the two created shorter tight paths. 
For this purpose, consider an ordering of the vertices $V(\PP_{i-1})=\{v_1,\dots,v_n\}$ according to $\PP_{i-1}$.
Let $$V^1_i=\big\{v_1,\dots,v_{\left\lfloor\frac{n}{2}\right\rfloor}\big\}\quad\text{ and }\quad V^2_i=\big\{v_{\left\lceil\frac{n}{2}\right\rceil+1},\dots,v_n\big\}.$$ 
Then $|V^1_i|=\left\lfloor\frac{n}{2}\right\rfloor=|V^2_i|$.
Now let $Z^1_i=E(\PP_{i-1}[V^1_i])$ and $Z^2_i=E(\PP_{i-1}[V^2_i])$. Clearly, these two sets form vertex-disjoint tight paths on $\left\lfloor\frac{n}{2}\right\rfloor$ vertices in $\GG$.
The size of $Z^1_i$ and $Z^2_i$ is
$$|Z^1_i|=|Z^2_i|=e(\PP^{(k)}_{\left\lfloor\tfrac{n}{2}\right\rfloor})=\left\lfloor\frac{n}{2}\right\rfloor-k+1\ge \frac{n-2k+1}{2}.$$

\noindent 
In the next step we show a key property of the edge sets $Z^{a}_{b}$ for $a\in[2]$, $b\in[i]$.
\\

\noindent \textbf{Claim.} Let $a_1,a_2\in[2]$, $b_1,b_2\in[i]$ such that $(a_1,b_1)\neq (a_2,b_2)$.
Then for any two edges $e_1\in N_{>q(i)}(Z^{a_1}_{b_1})$ and $e_2\in N_{>q(i)}(Z^{a_2}_{b_2})$ we have $$|e_1\cap e_2|<k-1.$$
\smallskip
\noindent \textit{Proof of the claim.} 
Assume that there are edges $e_1\in N_{>q(i)}(Z^{a_1}_{b_1})$, $e_2\in N_{>q(i)}(Z^{a_2}_{b_2})$ with $|e_1\cap e_2|\ge k-1$.
By definition, there is an edge $z_1\in Z^{a_1}_{b_1}$ such that $|e_1\cap z_1|>q(i)$ and an edge $z_2\in Z^{a_2}_{b_2}$ with $|e_2\cap z_2|>q(i)$.
\\

\begin{figure}[H]
\centering
\includegraphics[scale=0.65]{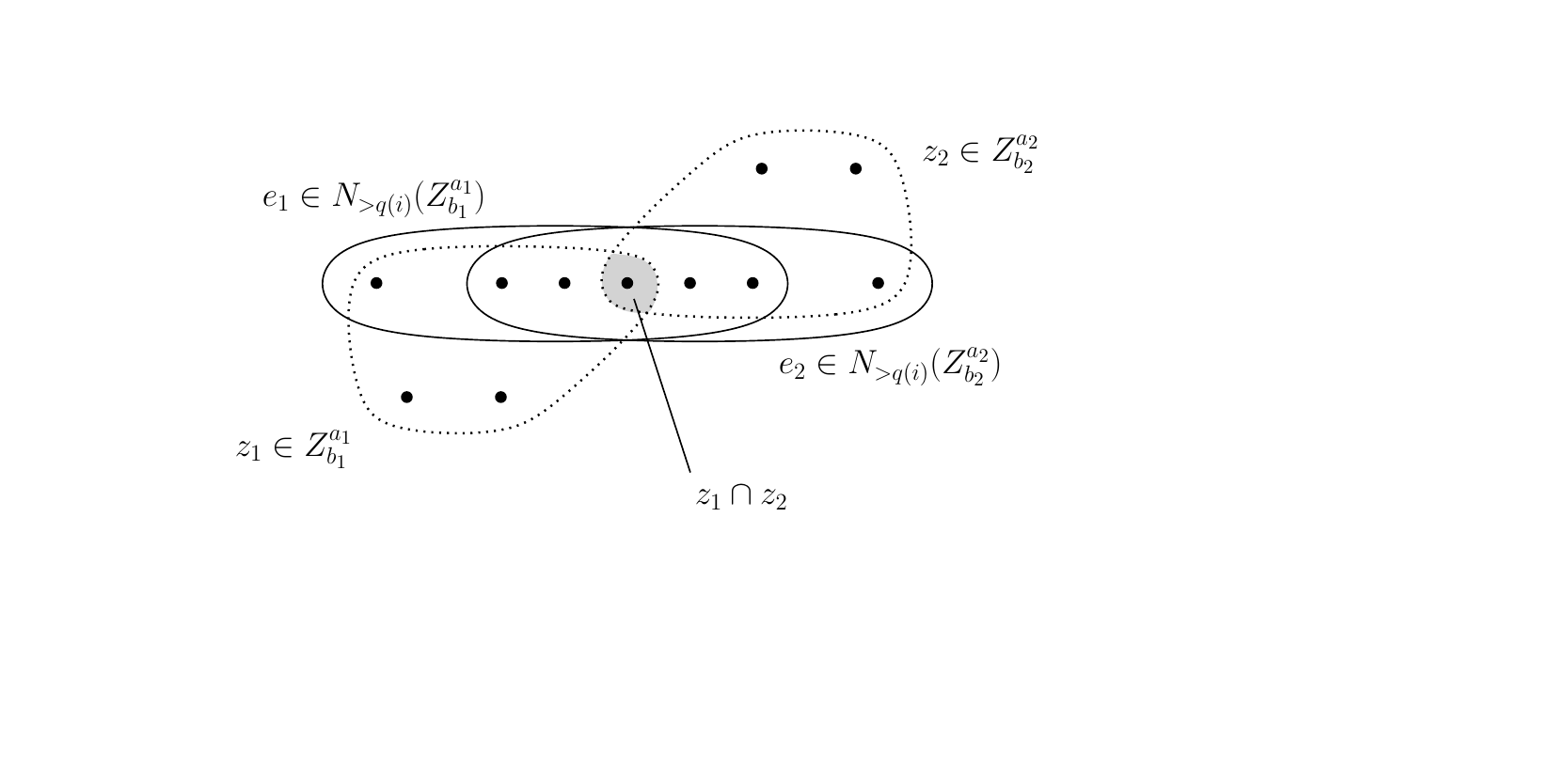}
\caption{Possible constellation of the edges in iteration step $i=1$ where $k=6$}
\end{figure}

\noindent We estimate the size of $z_1\cap z_2$ in order to find a contradiction to our assumption. 
Since $|e_1\cap e_2|\ge k-1$, we have $|e_1\backslash e_2|\le 1$ and so $|e_2\cap z_1|>q(i)-1$.
Applying this, we obtain:
\begin{align*}
|z_1\cap z_2|&\ge |e_2\cap z_1\cap z_2|\ge |e_2|-|e_2\backslash z_1|-|e_2\backslash z_2|= -|e_2|+|e_2\cap z_1|+|e_2\cap z_2|\\
&>-k+q(i)-1+q(i)=\left(1-\frac{1}{2^{i-1}}\right)(k+1)=q(i-1).
\end{align*}

If $b_1=b_2$, we have $\medcup Z^{a_1}_{b_1}\cap \medcup Z^{a_2}_{b_2}=\varnothing$ by construction. But then $q(i-1)<|z_1\cap z_2|=0$, which is a contradiction.

We suppose that $b_1\neq b_2$, then without loss of generality $b_1>b_2$ (and by this $b_1-1\ge1$).
By construction we know $z_1\in Z^{a_1}_{b_1}\subseteq E(\PP_{b_1-1})$. 
In the iteration step $b_1-1$ the tight path $\PP_{b_1-1}$ was chosen to be edge-disjoint from $\bigcup_{a\in[2],b<{b_1}} N_{>q(b_1-1)}(Z^a_b)$.
This yields that $z_1\notin N_{>q(b_1-1)}(Z^{a_2}_{b_2})$ and so
$$|z_1\cap z_2|\le q(b_1-1)\le q(i-1),$$
where the last inequality holds because $q$ is an increasing function, and we again reach a contradiction.
This concludes the proof of the claim. \hfill\qed\\


Now we find the next tight path $\PP_{i}$ in $\GG$ by considering the following coloring of $\GG$. 
For all $a\in[2]$ and $b\in[i]$, assign the color red to each edge in $N_{>q(i)}(Z^{a}_{b})$. The remaining edges are colored blue.
We will prove that there is a monochromatic blue $\PP^{(k)}_{n}$ in this coloring. We shall let $\PP_{i}$ be that path.
With this in mind, assume for a contradiction that there is a monochromatic red tight path $\RR$ on $n$ vertices in $\GG$.
\\

Clearly, each edge in $E(\RR)$ is in some neighborhood $N_{>q(i)}(Z^{a}_{b})$, $a\in[2], b\in[i]$.
Now the above claim provides that any two edges which are consecutive in $\RR$, so intersect in precisely $k-1$ vertices, 
belong to the same neighborhood $N_{>q(i)}(Z^{a}_{b})$ for some $a\in[2], b\in[i]$.
By repeating this argument, we obtain that $E(\RR)\subseteq N_{>q(i)}(Z^a_b)$ for some $a\in[2], b\in[i]$.
This implies that for all $e\in E(\RR)$, $$|e\cap \medcup Z^a_b|>q(i)\ge q(1)=\tfrac{k+1}{2}.$$
Then applying Proposition \ref{prop_gen} for the tight path $\RR$ and the vertex set $\medcup Z^a_b$ yields $|\medcup Z^a_b|>\frac{n}{2}$.
But by construction $Z^a_b$ forms a $k$-graph on precisely $\left\lfloor\frac{n}{2}\right\rfloor$ vertices, a contradiction.
\\

Consequently, there is no red tight path on $n$ vertices in the coloring, 
so the Ramsey property $\GG\rightarrow \PP^{(k)}_{n}$ implies the existence of a monochromatic blue $\PP^{(k)}_{n}$, which we denote $\PP_{i}$. 
Observe that for all $e\in E(\PP_{i})$ and for all $a\in[2], b\in[i]$, we have $e\notin N_{>q(i)}(Z^a_b)$, since all edges in these neighborhoods are colored in red.
\\

By iterating the described procedure for $i=1,\dots,\lambda$, we obtain edge sets $Z^a_b$ for $a\in[2]$, $b\in[\lambda]$ which are pairwise disjoint 
and additionally a tight path $\PP_{\lambda}$ on $n$ vertices such that each edge in $E(\PP_{\lambda})$ is not contained in any set $Z^a_b$.
This allows for the following estimate on the number of edges in $\GG$
\begin{align*}
e(\GG)&\ge \sum_{b\in[\lambda]} \big(|Z^1_b|+|Z^2_b|\big) + e(\PP_{\lambda})\ge\lambda (n-2k-1)+(n-k+1)\\
&\ge \left\lceil\log_2(k+1)\right\rceil\cdot n-(k-1)(2k+2)\ge \left\lceil\log_2(k+1)\right\rceil\cdot n-2k^2, 
\end{align*}
where in the last line we used $\left\lceil\log_2(k+1)\right\rceil\le k$.
\end{proof}
\smallskip

We point out that the above proof also applies to $3$-uniform tight paths, but does not yield an improvement of the trivial bound.
In order to obtain a refined bound in this case, we instead use a non-iterative adaption of the above proof.
\\

\begin{proof}[Proof of Theorem \ref{lb_tight2}]
Let $\GG$ be an arbitrary $3$-uniform hypergraph which has the Ramsey property $\GG\rightarrow \PP^{(3)}_{n}$. 
As before, we show that $\GG$ is a $3$-graph on at least $\tfrac{8}{3} n -\tfrac{28}{3}$ many edges.
Using the Ramsey property $\GG\rightarrow \PP^{(3)}_{n}$, there exists some tight path on $n$ vertices in~$\GG$. 
In particular, we find a shorter tight path $\PP_0$ on only $\left\lceil\tfrac{2}{3} n -\tfrac{7}{3}\right\rceil$ many vertices.
Observe that $e(\PP_0)= \left\lceil\tfrac{2}{3} n -\tfrac{7}{3}\right\rceil-2\ge \tfrac{2}{3}n-\tfrac{13}{3}$.
\\

In order to find a tight path $\PP_1$ which is edge-disjoint from $\PP_0$, we consider the following coloring.
Color all edges in the $1$-neighborhood $N_{>1}\big(E(\PP_0)\big)$ in red and the remaining edges in blue. 
Assume for a contradiction that in this coloring there is a monochromatic red tight path on $n$ vertices, say $\RR$.
Then Proposition \ref{prop_gen} applied to the tight path $\RR$ and the vertex set $V(\PP_0)$ provides a contradiction.
Since $\GG\rightarrow \PP^{(3)}_{n}$, there is a monochromatic blue tight path on $n$ vertices in~$\GG$.
This implies that there is also a blue tight path on $n-1$ vertices, i.e.\ on $n-3$ edges.
We fix such a tight path $\PP_1$ with $e(\PP_1)=n-3$. Note that $N_{>1}\big(E(\PP_0)\big)$ and $E(\PP_1)$ are disjoint edge sets.
\\

In the following, in order to find a third edge-disjoint tight path, we consider another coloring of $\GG$. 
From now on, let each edge in $E(\PP_0) \cup E(\PP_1)$ be colored red and all other edges blue.
Assume for a contradiction that there is a red tight path $\RR$ on $n$ vertices in this coloring. 
Then neither $E(\RR)\subseteq E(\PP_0)$ nor $E(\RR)\subseteq E(\PP_1)$, because the two edge sets have size strictly less than $e(\PP^{(3)}_{n})$. 
Therefore, $\RR$ consists of edges of both $E(\PP_0)$ and $E(\PP_1)$. 
Both of these edge sets are disjoint, so there exist two edges $e_1\in E(\PP_0)\cap E(\RR),e_2\in E(\PP_1)\cap E(\RR)$ which are consecutive in $\RR$, i.e.\ $|e_1\cap e_2|=2$. 
But that is a contradiction to the fact that $N_{>1}\big(E(\PP_0)\big)$ and $E(\PP_1)$ are disjoint.
Consequently, there is no red $\PP^{(3)}_{n}$ in this coloring.
By the same argument as before, there is a blue tight path $\PP_2$ on $n$ vertices in $\GG$.

Then the three edge sets $E(\PP_0), E(\PP_1), E(\PP_2)$ are pairwise disjoint. Thus,
\begin{align*}
e(\GG)\ge e(\PP_0)+e(\PP_1)+e(\PP_2)\ge \tfrac{8}{3} n -\tfrac{28}{3}.&\qedhere
\end{align*}
\end{proof}


\section*{Acknowledgments}
The author would like to thank Mathias Schacht for helpful discussions and support with the thesis on which this paper is based, Maria Axenovich for comments on the manuscript and her help with its final polish as well as the two thorough referees for their careful reading of the paper and their useful comments.

\end{document}